\documentclass[12pt]{article}
\usepackage{etoolbox}
\usepackage{dynkin-diagrams}
\def\row#1/#2!{#1_{\IfStrEq{#2}{}{n}{#2}} & \dynkin{#1}{#2}\\}

\usepackage{}
\usepackage[numbers,sort&compress]{natbib}

\usepackage{color}
\usepackage{threeparttable}

\usepackage[english]{babel}
\usepackage{t1enc}
\usepackage{amsthm,amsmath,amssymb}
\usepackage{mathrsfs}
\usepackage[mathscr]{eucal}
\usepackage[latin1]{inputenc}
\usepackage[english]{babel}
\usepackage{latexsym}
\usepackage{graphicx}
\usepackage[noend]{algpseudocode}
\usepackage{algorithmicx,algorithm}
\usepackage{lineno}
\usepackage{booktabs}
\usepackage{multirow}
\newtheorem{theorem}{Theorem}
\newtheorem{corollary}{Corollary}

\newtheorem{lemma}{Lemma}

\newtheorem{conjecture}{Conjecture}

\def \T{\textup{T}}
\def \Disc{\textup{Disc}}
\def \rank{\textup{rank}}

\def \diag{\textup{diag~}}

\makeatletter
\newcommand{\rmnum}[1]{\romannumeral #1}
\newcommand\restr[2]{{
		\left.\kern-\nulldelimiterspace 
		#1 
		\right|_{#2} 
}}
\newcommand{\Rmnum}[1]{\expandafter\@slowromancap\romannumeral #1@}
 
\makeatother
\title{On the walk matrix of the Dynkin graph $D_n$}
\author{\small  Wei Wang\thanks{Corresponding author: wangwei.math@gmail.com}\quad\quad Chuanming Wang\quad\quad Songlin Guo
\\
{\footnotesize$^{\rm a}$School of Mathematics, Physics and Finance, Anhui Polytechnic University, Wuhu 241000, P. R. China}
}
\date{}
\begin{document}
	\maketitle
	
	\begin{abstract}
		Let $W(D_n)$ denote the walk matrix of the Dynkin graph $D_n$, a tree obtained from the path of order $n-1$  by adding
		a pendant edge at the second vertex. We prove that $\rank\,W(D_n)=n-2$ if $4\mid n$ and $\rank\,W(D_n)=n-1$ otherwise. Furthermore, we prove that the Smith normal form of $W(D_n)$ is $$\diag[\underbrace{1,1,\ldots,1}_{\lceil\frac{n}{2}\rceil},\underbrace{2,2,\ldots,2}_{\lfloor\frac{n}{2}\rfloor-1},0]$$  when $4\nmid n$. This confirms a recent conjecture in [W.~Wang, F.~Liu, W.~Wang, Generalized spectral characterizations of almost controllable graphs, European J. Combin., 96(2021):103348].\\
		
		
	\end{abstract}
	\section{Introduction}
	Let $G$ be a simple $n$-vertex graph with adjacency matrix $A$. The \emph{walk matrix} of $G$ is
	\begin{equation}
	W(G):=[e,Ae,\ldots,A^{n-1}e],
	\end{equation}
	where $e$ is the all-ones vector of dimension $n$.
	
	Let $D_n$ ($n\ge 4$) be a family of graphs as shown in Fig.~1. Let $\hat{W}(D_n)$ denote the $(n-1)\times (n-1)$ matrix obtained from $W(D_n)$ by removing the first row and the last column. For example, when $n=5$, we have
	\begin{equation}
	W(D_5)=\begin{pmatrix}
1 & 1 & 3 & 4 & 10 \\
1 & 1 & 3 & 4 & 10 \\
1 & 3 & 4 & 10 & 14 \\
1 & 2 & 4 & 6 & 14 \\
1 & 1 & 2 & 4 & 6 \\
	\end{pmatrix}\quad\text{and}\quad
		\hat{W}(D_5)=\begin{pmatrix}
1 & 1 & 3 & 4 \\
1 & 3 & 4 & 10 \\
1 & 2 & 4 & 6 \\
1 & 1 & 2 & 4 \\
	\end{pmatrix}.
	\end{equation}
	\begin{conjecture}[\cite{wang2021Eujc}]\label{main}
			\begin{equation}\label{conequ1}
		\det \hat{W}(D_n)=
		\begin{cases}
		\pm 2^{\lfloor\frac{n}{2}\rfloor-1} &\text{if~$4\nmid n$},\\
		0 & \text{if~$4\mid n$.}
		\end{cases}
		\end{equation}
	\end{conjecture}
	\begin{figure}
	\centering
	\includegraphics[height=2cm]{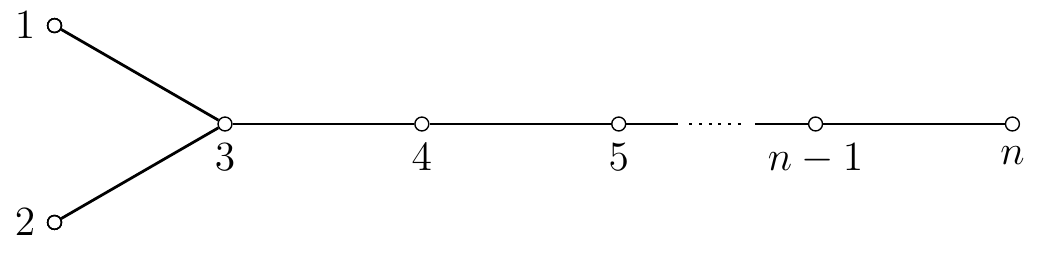}
	\caption{Graph $D_n$}
\end{figure}
The main aim of this paper is to confirm Conjecture \ref{main}. Indeed, we shall prove a more refined result.
\begin{theorem}\label{mainthm}
	We have
	
\textup{(\rmnum{1})} 	If $4\nmid n$ then $\det \hat{W} (D_n)=\pm 2^{\lfloor\frac{n}{2}\rfloor-1}$.
	
	\textup{(\rmnum{2})} If $4\mid n$ then $\rank\,\hat{W} (D_n)=n-2$.
\end{theorem}
For any integral $n\times n$ matrix $M$, there exist two unimodular matrices $U$ and $V$ such that $UMV$ is a diagonal matrix $\diag[d_1,d_2,\ldots,d_n]$ with $d_i\mid d_{i+1}$ for $i=1,2,\ldots,n-1$.
The diagonal matrix $\diag[d_1,d_2,\ldots,d_n]$ is called the \emph{Smith normal form} of $M$ and $d_i$ is called the $i$-th invariant factor of $M$. Note that the Smith normal form is a refinement of the determinant and the rank since $\det M=\pm d_1d_2\cdots d_n$ and $\rank\,M=\max\{i\colon\,d_i\neq 0\}$.

We shall see that the $n\times n$ matrix $W(D_n)$ and the $(n-1)\times (n-1)$ submatrix $\hat{W}(D_n)$ have almost the same Smith normal form. These two matrices have the same $i$-th invariant for $i=1,2,\ldots,n-1$ and the $n$-th (the last) invariant of $W(D_n)$ is zero. Thus we can easily determine $\rank\, W(D_n)$ by Theorem \ref{mainthm}:
\begin{equation}\label{rkW}
\rank\,W(D_n) =
\begin{cases}
n-1  &  \text{if $4\nmid n$,} \\
n-2
&  \text{if $4\mid n$.}
\end{cases}
\end{equation}
For the case $4\nmid n$, we can obtain the Smith normal form of $W(D_n)$ by Theorem \ref{mainthm}(\rmnum{1}).
\begin{theorem}\label{snfd}
	If $4\nmid n$ then $W(D_n)$  has the Smith normal form $$\diag[\underbrace{1,1,\ldots,1}_{\lceil\frac{n}{2}\rceil},\underbrace{2,2,\ldots,2}_{\lfloor\frac{n}{2}\rfloor-1},0].$$
\end{theorem}

 An eigenvalue of $G$ (i.e. of $A$) is said to be \emph{main} if $A$ has an associated eigenvector which is not orthogonal to $e$. It is known that the number of main eigenvalues of $G$ equals the rank of $W(G)$; see \cite{hagos2002}. Thus, the following result on the number of main eigenvalues of $D_n$ is an immediate consequence of \eqref{rkW}.
 \begin{corollary}
 	The number of main eigenvalues of $D_n$ is $n-1$ if $4\nmid n$, and is $n-2$ otherwise.
 \end{corollary}
	\section{Preliminaries}
	We use the method of equitable partitions of graphs. We refer to \cite{cvetkovic2010} for relevant terminologies.  Let $\Pi=\{\{1,2\},\{3\},\{4\},\ldots,\{n\}\}$ be a partition of $V(D_n)$ which is clearly equitable,  and let
	$$\small{C=\begin{pmatrix}
		1&{}&{}&{}&{}\\
		1&{}&{}&{}&{}\\
		{}&1&{}&{}&{}\\
		{}&{}&1&{}&{}\\
		{}&{}&{}&\ddots&{}\\
		{}&{}&{}&{}&1
		\end{pmatrix}}_{n\times(n-1)}$$ be the characteristic matrix corresponding to $\Pi$. The divisor matrix of $\Pi$ is
	$$\small{B=\begin{pmatrix}
		0&1&{}&{}&{}\\
		2&0&1&{}&{}\\
		{}&1&\ddots&\ddots&{}\\
		{}&{}&\ddots&\ddots&1\\
		{}&{}&{}&1&0\\
		\end{pmatrix}}_{(n-1)\times (n-1)}.$$
	For convenience, for an $m\times m$ matrix $M$, we define $W(M)=[e_m,Me_m,\ldots,M^{m-1}e_m]$, where the subscript $m$ refers to the number of entries of the all-ones vector $e_m$. We shall call $W(M)$ the walk matrix of $M$.
	
	The following simple fact is the starting point of our argument.
	\begin{lemma}\label{cofdet}
		$\hat{W}(D_n)=W(B).$
	\end{lemma}
	\begin{proof}
		Let $A=A(D_n)$ be the adjacency matrix of $D_n$. Then we have $AC=CB$ and hence $A^k C=CB^k$ for any $k\ge 0$. Noting that $Ce_{n-1}=e_n$, we have $A^k e_n=CB^ke_{n-1}$ and hence
		$$[e_n,Ae_n,\ldots,A^{n-2}e_n]=CW(B).$$
		Note that removing the first row from $[e_n,Ae_n,\ldots,A^{n-2}e_n]$ and $C$ results in $\hat{W}(D_n)$ and the identity matrix, respectively. Thus, $\hat{W}(D_n)=W(B)$, as desired.	
	\end{proof}
	The following lemma is essentially due to Mao, Liu and Wang~\cite{mao2015LAA}.
	\begin{lemma}\label{relWA}
		Let $M$ (and hence $M^\T$) be a real $m\times m$  matrix which is diagonalizable over the real field $\mathbb{R}$. Let  $\xi_1,\xi_2,\ldots,\xi_m$ be  $m$ independent eigenvectors of $M^\T$ corresponding to  eigenvalues $\lambda_1,\lambda_2,\ldots,\lambda_m$, respectively. Then we have
		$$\det W(M)=\frac{\prod_{1\le k<j\le m} (\lambda_j-\lambda_k)\prod_{j=1}^m e^\T\xi_j}{\det[\xi_1,\xi_2,\ldots,\xi_m]}.$$
		Moreover, if $\lambda_1,\lambda_2,\ldots, \lambda_m$ are pairwise different, then $\rank\,W(M)=|\{j\colon\,1\le j\le m ~and~e^\T \xi_j\neq 0\}|$.
	\end{lemma}
	\begin{proof}
		Write $W=W(M)$ and $e=e_m$. For each $s\in\{1,2,\ldots,m-1\}$, we have
		$$e^\T (M^\T)^s(\xi_1,\ldots,\xi_m)=(\lambda_1^se^\T \xi_1,\ldots,\lambda_m^se^\T\xi_m)=(\lambda_1^s,\ldots,\lambda_m^s)\left(\begin{matrix}
		e^\T\xi_1&&\\
		&\ddots&\\
		&&e^\T\xi_m\\
		\end{matrix}
		\right),$$
		and hence
		$$
		W^\T(\xi_1,\ldots,\xi_m)=\left(\begin{matrix} e^\T\\e^\T M^\T\\ \vdots\\e^\T (M^\T)^{m-1}\\\end{matrix}\right)(\xi_1,\ldots,\xi_m)=\left(\begin{matrix} 1&\ldots&1\\\lambda_1&\ldots&\lambda_m\\ \vdots&\cdots&\vdots\\ \lambda_1^{m-1}&\ldots&\lambda_m^{m-1}\\\end{matrix}\right)\left(\begin{matrix}
		e^\T\xi_1&&\\
		&\ddots&&\\
		&&e^\T\xi_m\\
		\end{matrix}
		\right).
		$$
		Therefore, $\det W^\T\det [\xi_1,\ldots,\xi_m]=\prod_{1\le k<j\le m} (\lambda_j-\lambda_k)\prod_{j=1}^m e^\T\xi_j.$
		This proves the first assertion as $\det W^\T=\det W$. Furthermore, assuming that all $\lambda_j$'s are pairwise different, we have $\rank\, W^\T=\rank\,\diag[e^\T\xi_1,\ldots,e^\T \xi_m]$ and hence the second assertion follows.
	\end{proof}
\section{Chebyshev polynomials and some identities}
	Let $T_n(x)$ and $U_n(x)$ be the Chebyshev polynomials of the first and the second kind respectively. 	The \emph{discriminant} of a polynomial $P(x)$ of degree $n$ and leading coefficient $a_n$ is
	\begin{equation}\label{discP}
	\Disc_x P(x) = a_n^{2n-2}\prod_{1\le k<j\le n}(r_j-r_k)^2,
	\end{equation}
	where $r_1,r_2,\ldots,r_n$ are the roots of $P(x)$.

	The following  lemma summarizes some basic properties of $T_n(x)$ and $U_n(x)$ that will be used in this paper.
	\begin{lemma}[\cite{rivlin}]\label{basiccheb}
	Let $n\ge 1$. Then $T_n(x)$ and $U_n(x)$ have the following properties:\\	
	\textup{(\rmnum{1})} The roots of $T_n(x)$ and $U_n(x)$ are $\{\cos\frac{2j-1}{2n}\pi\colon\, j=1,\ldots,n\}$ and $\{\cos\frac{j}{n+1}\pi\colon\,j=1,\ldots,n\}$, respectively;\\
	\textup{(\rmnum{2})} The leading coefficients of $T_n(x)$ and $U_n$  are $2^{n-1}$ and $2^n$, respectively;\\
	\textup{(\rmnum{3})} $T_n(x)$ and $U_n(x)$	have the same constant term $b_n$, where
		$$b_n=\begin{cases}
	(-1)^{\frac{n}{2}}&\text{~if~$n$~is even},\\
	0&\text{~if~$n$~is odd}.
	\end{cases}
	$$
	\textup{(\rmnum{4})} $\Disc_x T_n(x)=2^{(n-1)^2}n^{n}.$
		\end{lemma}
	\begin{lemma}
		\label{forfirst}
		\begin{equation}\label{fromdiscT}
		\prod_{1\le k<j\le m}\left(2\cos\frac{(2j-1)\pi}{2m}-2\cos\frac{(2k-1)\pi}{2m}\right)=\pm2^\frac{m-1}{2}m^{\frac{m}{2}}, ~\text{for}~m\ge 1.
		\end{equation}
	\end{lemma}
	\begin{proof}
		The leading coefficient of $T_m(x)$ is $2^{m-1}$ and the roots are $\cos\frac{(2k-1)\pi}{2m}$, $k=1,2,\ldots,m$. Thus, by  \eqref{discP} and Lemma \ref{basiccheb}(\rmnum{4}), we have
		\begin{equation}
		2^{(m-1)(2m-2)}	\prod_{1\le k<j\le m}\left(\cos\frac{(2j-1)\pi}{2m}-\cos\frac{(2k-1)\pi}{2m}\right)^2=2^{(m-1)^2} m^m,
		\end{equation}
		which implies \eqref{fromdiscT}.
	\end{proof}
	The following two formulas are direct consequences of Vieta's theorem for $T_{4m}(x)$ and $U_{4m}(x)$.
	\begin{lemma}\label{pc8m}
		$$\prod_{j=1}^{4m}\cos\frac{2j-1}{8m}\pi=2^{-4m+1}.$$
	\end{lemma}

	\begin{lemma}\label{pc4mp1}
		$$	\prod_{j=1}^{4m}\cos\frac{j}{4m+1}\pi=2^{-4m}.$$
	\end{lemma}
	\begin{lemma}\label{sumcos}
		$$\sum_{k=1}^{m}\cos(ak+b)x=\frac{1}{2\sin\frac{1}{2}ax}(\sin((m+\frac{1}{2})a+b)x-\sin(\frac{1}{2}a+b)x),~ \text{for}~ \sin \frac{1}{2}ax\neq 0.$$
		\end{lemma}
	\begin{proof}
		Using the formula $\cos \alpha \sin \beta=\frac{1}{2}(\sin(\alpha+\beta)-\sin(\alpha-\beta))$, we have
		\begin{eqnarray*}
			\sum_{k=1}^{m}\cos(ak+b)x\sin \frac{1}{2}ax&=&\frac{1}{2}\sum_{k=1}^{m}\sin((k+\frac{1}{2})a+b)x-\sin((k-\frac{1}{2})a+b)x\\
			&=&\frac{1}{2}(\sin((m+\frac{1}{2})a+b)x-\sin(\frac{1}{2}a+b)x).
			\end{eqnarray*}
		This proves Lemma \ref{sumcos}.

	\end{proof}
	
	\begin{corollary}\label{prodsin}
		$$\prod_{j=1}^{m-1}\sin\frac{2k-1}{4(m-1)}\pi=2^{\frac{3}{2}-m}.$$
	\end{corollary}
\begin{proof}
	Using Vieta's theorem for $T_{2(m-1)}$, we have
	$$\prod_{j=1}^{2m-2}\cos\frac{2j-1}{4(m-1)}\pi=(-1)^{m-1}2^{3-2m}.$$
	Define $\sigma(j)=2m-1-j$. Using the formula $\cos(\pi-x)=-\cos(x)$, we  find that $\cos\frac{2\sigma(j)-1}{4(m-1)}\pi=-\cos\frac{2j-1}{4(m-1)}\pi$ and hence
		$$\prod_{j=1}^{2m-2}\cos\frac{2j-1}{4(m-1)}\pi=\prod_{j=1}^{m-1}\cos\frac{2j-1}{4(m-1)}\pi\cos\frac{2\sigma(j)-1}{4(m-1)}\pi=(-1)^{m-1}\prod_{j=1}^{m-1}\cos^2\frac{2j-1}{4(m-1)}\pi.$$
		Thus, we have
			$$\prod_{j=1}^{m-1}\cos\frac{2j-1}{4(m-1)}\pi=\pm 2^{\frac{3}{2}-m},$$
			where the sign `$\pm$' is indeed a `$+$' as each term in the product is positive. By the formula $\cos(x)=\sin(\frac{\pi}{2}-x)$, we have
			$$\prod_{j=1}^{m-1}\cos\frac{2j-1}{4(m-1)}\pi=\prod_{j=1}^{m-1}\sin\frac{2(m-j)-1}{4(m-1)}\pi=\prod_{j=1}^{m-1}\sin\frac{2j-1}{4(m-1)}\pi.$$
			This completes the proof.
\end{proof}

\section{Proofs of the main results}

\begin{lemma}\label{eigB}
	Let $\xi_k=(1,\cos\frac{(2k-1)\pi}{2(n-1)},\cos\frac{(2k-1)2\pi}{2(n-1)},\ldots,\cos\frac{(2k-1)(n-2)\pi}{2(n-1)})^\T$ and $\lambda_k=2\cos\frac{(2k-1)\pi}{2(n-1)}$ for $k=1,2,\ldots,n-1$. Then $B^\T \xi_k=\lambda_k\xi_k$.
\end{lemma}
\begin{proof}
We fix $k$ and	write $\theta=\frac{(2k-1)\pi}{2(n-1)}$. Define $\eta=(a_0,a_1,\ldots,a_{n})$ where $a_i=\cos (i-1)\theta$ for $i=0,1,\ldots,n$. By the formula
$\cos\alpha+\cos\beta=2\cos\frac{\alpha+\beta}{2}\cos\frac{\alpha-\beta}{2}$, one easily sees that
\begin{equation}\label{flx}
a_{s-1}+a_{s+1}=\cos(s-2)\theta+\cos s\theta=2\cos (s-1)\theta \cos\theta=\lambda_k a_s,~ s=1,2,\ldots,n-1.
\end{equation}
 Note that $a_{0}=a_2$ and $a_{n}=\cos\frac{(2k-1)\pi}{2}=0$. We can rewrite \eqref{flx} in matrix form as follows:
 	$$\small{\begin{pmatrix}
 	0&2&{}&{}&{}\\
 	1&0&1&{}&{}\\
 	{}&1&\ddots&\ddots&{}\\
 	{}&{}&\ddots&\ddots&1\\
 	{}&{}&{}&1&0\\
 	\end{pmatrix}}\normalsize{\begin{pmatrix}a_1\\a_2\\a_3\\\vdots\\a_{n-1}\end{pmatrix}=\lambda_k\begin{pmatrix}a_1\\a_2\\a_3\\\vdots\\a_{n-1}\end{pmatrix}},$$
 that is, $B^\T \xi_k=\lambda_k\xi_k$. This completes the proof.
\end{proof}

 \begin{lemma}\label{fordet}
 	$\det[\xi_1,\xi_2,\ldots,\xi_{n-1}]=\pm2^{-\frac{n-2}{2}}(n-1)^\frac{n-1}{2}$, where $\xi_1,\xi_2,\ldots,\xi_{n-1}$ are defined as in Lemma \ref{eigB}.
 \end{lemma}

\begin{proof}
	Let $M=[\xi_1,\xi_2,\ldots,\xi_{n-1}]$ and $\eta_i$ be the $i$-th row of $M$. That is, $$\eta_i=\left(\cos\frac{(i-1)\pi}{2(n-1)},\cos\frac{3(i-1)\pi}{2(n-1)},\cos\frac{5(i-1)\pi}{2(n-1)},\ldots,\cos\frac{(2n-3)(i-1)\pi}{2(n-1)}\right).$$
	
\noindent	\emph{Claim} 1. $\eta_1\eta_1^\T=n-1$ and  $\eta_i\eta_i^\T=\frac{n-1}{2}$ for $i=2,3,\ldots,n-1$.

Note that $\eta_1=e_{n-1}$. We have $\eta_1\eta_1^\T=n-1$. Now assume $i\in \{2,3,\ldots,n-1\}$.
Write $\alpha=\frac{(i-1)\pi}{2(n-1)}$. Then $0<\alpha<\pi/2$ and hence $\sin 2\alpha \neq 0$.  Using the formulas $\cos^2 x=\frac{1}{2}(1+\cos 2x)$ and Lemma \ref{sumcos},  we have
\begin{eqnarray*}
\eta_i\eta_i^\T&=&\sum_{k=1}^{n-1}\cos^2 (2k-1)\alpha\\
	&=&\frac{n-1}{2}+\frac{1}{2}\sum_{k=1}^{n-1}\cos (4k-2)\alpha\\
	&= &\frac{n-1}{2}+\frac{1}{4\sin2\alpha}\sin(4n-4)\alpha\\
		&= &\frac{n-1}{2}+\frac{1}{4\sin2\alpha}\sin2(i-1)\pi\\
	&=&\frac{n-1}{2}.
\end{eqnarray*}
This proves Claim 1.

\noindent\emph{Claim} 2. $\eta_i\eta_j^\T=0$ for distinct $i,j$ in $\{1,2,\ldots,n-1\}$.

Write $\alpha=\frac{(i-1)\pi}{2(n-1)}$ and $\beta=\frac{(j-1)\pi}{2(n-1)}$. As $i$ and $j$ are distinct integers in $\{1,2,\ldots,n-1\}$, we see that $\alpha+\beta\in(0,\pi)$ and $\alpha-\beta\in(-\frac{\pi}{2},0)\cup(0,\frac{\pi}{2})$. Thus, both $\sin(\alpha+\beta)$ and $\sin(\alpha-\beta)$ are nonzero. Using a similar argument as above, we have
\begin{eqnarray*}
	\eta_i\eta_j^\T&=&\sum_{k=1}^{n-1}\cos (2k-1)\alpha\cdot\cos (2k-1)\beta\\
	&=&\frac{1}{2}\sum_{k=1}^{n-1}\left(\cos(2k-1)(\alpha+\beta)+\cos(2k-1)(\alpha-\beta)\right)\\
	&= &\frac{1}{4\sin(\alpha+\beta)}\sin 2(n-1)(\alpha+\beta)+\frac{1}{4\sin(\alpha-\beta)}\sin2(n-1)(\alpha-\beta)\\
	&=&\frac{1}{4\sin(\alpha+\beta)}\sin(i+j-2)\pi+\frac{1}{4\sin(\alpha-\beta)}\sin(i-j)\pi\\
	&=&0.
\end{eqnarray*}
Claim 2 follows.

By Claims 1 and 2, we have $$MM^\T=(\eta_i\eta_j^\T)=\diag[n-1,\underbrace{(n-1)/2,(n-1)/2,\ldots,(n-1)/2}_{n-2}].$$
It follows that
$$\det M=\pm\sqrt{\det MM^\T}=\pm2^{-\frac{n-2}{2}}(n-1)^\frac{n-1}{2}.$$ This completes the proof of Lemma \ref{fordet}.
\end{proof}
\begin{lemma}\label{etxi}
Let $\xi_j$ be defined as in Lemma \ref{eigB}. If $n\not\equiv 0\pmod{4}$ then $\prod_{j=1}^{n-1}e^\T \xi_j=\pm 2^{1-\lceil\frac{n}{2}\rceil}\neq 0$, otherwise $e^\T \xi_{\frac{n}{2}}=0$ and  $e^\T \xi_{j}\neq 0$ for each $j\in\{1,2,\ldots,n-1\}\setminus\{\frac{n}{2}\}$.
\end{lemma}
\begin{proof}
	Write $\alpha_j=\frac{(2j-1)\pi}{2(n-1)}$ for $j\in\mathbb{Z}$. Then $\xi_j=(1,\cos \alpha_j,\cos2\alpha_j,\ldots,\cos(n-2)\alpha_j)^\T$ for $j=1,2,\ldots,n-1$. Clearly, $\sin\frac{1}{2}\alpha_j\neq0$. By Lemma \ref{sumcos}, we have
	$$e^\T \xi_j=\sum_{k=1}^{n-1}\cos(k-1)\alpha_j=\frac{1}{2\sin\frac{1}{2}\alpha_j}\left(\sin(n-\frac{3}{2})\alpha_j+\sin\frac{1}{2}\alpha_j\right).$$
	Consequently, by the formula $\sin x+\sin y=2\sin\frac{1}{2}(x+y)\cos\frac{1}{2}{(x-y)}$, we have
	\begin{equation}\label{exi}
e^\T \xi_j=\frac{1}{\sin\frac{1}{2}\alpha_j}{ \sin\frac{n-1}{2}\alpha_j\cos\frac{n-2}{2}\alpha_j}.
	\end{equation}
	It follows that
	\begin{equation}\label{prdexi}
\prod_{j=1}^{n-1}e^\T\xi_j=\frac{\prod_{j=1}^{n-1}\sin\frac{n-1}{2}\alpha_j\prod_{j=1}^{n-1}\cos\frac{n-2}{2}\alpha_j}{\prod_{j=1}^{n-1}\sin\frac{1}{2}\alpha_j}.
	\end{equation}
	As $\frac{n-1}{2}\alpha_j=\frac{(2j-1)\pi}{4}$, we find that $\sin\frac{n-1}{2}\alpha_j=\pm \frac{\sqrt{2}}{2}$ and hence
	\begin{equation}\label{psi}
\prod_{j=1}^{n-1}\sin\frac{n-1}{2}\alpha_j=\pm 2^{-\frac{n-1}{2}}.
	\end{equation}
Noting $\frac{1}{2}\alpha_j=\frac{2j-1}{4(n-1)}\pi$ and using Corollary \ref{prodsin}, we have
\begin{equation}\label{psi2}
\prod_{j=1}^{n-1}\sin\frac{1}{2}\alpha_j=\prod_{j=1}^{n-1}\sin\frac{2j-1}{4(n-1)}\pi=2^{\frac{3}{2}-n}.
\end{equation}
Plugging \eqref{psi} and \eqref{psi2} into \eqref{exi}, we get
	\begin{equation}\label{exisimp}
\prod_{j=1}^{n-1}e^\T\xi_j=2^{\frac{n}{2}-1}\prod_{j=1}^{n-1}\cos\frac{n-2}{2}\alpha_j.
\end{equation}
\noindent\emph{Claim} 1. $\prod_{j=1}^{n-1}\cos\frac{n-2}{2}\alpha_j=\pm 2^{2-n}$ for $n\equiv 2\pmod{4}$.

Write $n=4m+2$. Then we have
$$\frac{n-2}{2}\alpha_j=\frac{m(2j-1)}{4m+1}\pi.$$
Let $S=\{m(2j-1)\colon\,1\le j\le 4m+1\}$.
As $\gcd(m,4m+1)=1$ and $\gcd(2,4m+1)=1$, one easily sees that $S$ is  a complete system of residues modulo $4m+1$. Thus we have
$$\prod_{j=1}^{n-1}\cos\frac{n-2}{2}\alpha_j=\prod_{s\in S}\cos\frac{s}{4m+1}\pi=\pm \prod_{j=0}^{4m}\cos\frac{j}{4m+1}\pi=\pm\prod_{j=1}^{4m}\cos\frac{j}{4m+1}\pi.$$
Claim 1 follows by Lemma \ref{pc4mp1}.

\noindent\emph{Claim} 2. $\prod_{j=1}^{n-1}\cos\frac{n-2}{2}\alpha_j=\pm 2^{\frac{3}{2}-n}$ for odd $n$.

 Note that $\alpha_{1-j}=-\alpha_j$. We have
 $$\prod_{j=1}^{n-1}\cos\frac{n-2}{2}\alpha_j=\prod_{j=1}^{n-1}\cos\frac{n-2}{2}\alpha_{1-j}=\prod_{j=-n+2}^{0}\cos\frac{n-2}{2}\alpha_{j},$$
and hence
\begin{equation}\label{pmsq}
 \prod_{j=1}^{n-1}\cos\frac{n-2}{2}\alpha_j=\pm\left(\prod_{j=-n+2}^{n-1}\cos\frac{n-2}{2}\alpha_{j}\right)^\frac{1}{2}.
\end{equation}

Write $n=2m+1$. Then we have
$$\frac{n-2}{2}\alpha_j=\frac{(2m-1)(2j-1)}{8m}\pi.$$
For $j\in\{-n+2,-n+3,\ldots,n-1\}=\{-2m+1,-2m+2,\ldots,2m\}$, let $r_j$ be the least nonnegative residue of $(2m-1)(2j-1)$  modulo $8m$. Let $j$ and $j'$ be any two distinct integers in $\{-2m+1,-2m+2,\ldots,2m\}$. Then $j\not
\equiv j'\pmod{4m}$ and hence $2j-1\not\equiv 2j'-1\pmod{8m}$. Consequently, $(2m-1)(2j-1)\not\equiv (2m-1)(2j'-1)\pmod{8m}$ as $2m-1$ and $8m$ are coprime. This means that $r_j\neq r_j'$. Note that each $r_j$ is odd and there are exactly $4m$ odd numbers in $\{0,1,\ldots,8m-1\}$. We must have
\begin{equation}\label{eqres}
\{r_{-2m+1},r_{-2m+2},\ldots,r_{2m}\}=\{1,3,\ldots,8m-1\}.
\end{equation}
As
$$\frac{n-2}{2}\alpha_j-\frac{r_j}{8m}\pi=\frac{(2m-1)(2j-1)-r_j}{8m}\pi=k\pi$$ for some $k\in\mathbb{Z}$, we have $\cos\frac{n-2}{2}\alpha_j=\pm \cos\frac{r_j}{8m}\pi.$ It follows from \eqref{eqres} that
$$\prod_{j=-n+2}^{n-1}\cos\frac{n-2}{2}\alpha_j=\pm \prod_{\substack{1\le j\le 8m-1\\ j\text{~odd}}}\cos\frac{j}{8m}\pi=\pm \prod_{j=1}^{4m}\cos\frac{2j-1}{8m}\pi.$$
Consequently, by \eqref{pmsq} and Lemma \ref{pc8m}, we have
\begin{equation}\label{pm}
\prod_{j=1}^{n-1}\cos\frac{n-2}{2}\alpha_j=\pm\left|\prod_{j=1}^{4m}\cos\frac{2j-1}{8m}\pi\right|^\frac{1}{2}=\pm 2^{-2m+\frac{1}{2}}=\pm 2^{\frac{3}{2}-n}.
\end{equation}
This proves Claim 2.

\noindent\emph{Claim} 3. Let $n=4m$ and $j\in\{1,2,\ldots,4m-1\}$. Then $e^\T \xi_j=0$ if and only if $j=2m$.

By \eqref{exi} and \eqref{psi}, we see that $e^\T \xi_j=0$ if and only if $\cos \frac{n-2}{2}\alpha_j=0$. As
$$\frac{n-2}{2}\alpha_j= \frac{(2m-1)(2j-1)}{2(4m-1)}\pi,$$
we find that $\cos \frac{n-2}{2}\alpha_j=0$ if and only if  $\frac{(2m-1)(2j-1)}{(4m-1)}$ is an odd integer. Write $h(j)=\frac{(2m-1)(2j-1)}{(4m-1)}$. Then $h(2m)=2m-1$, which is an odd integer. If $j\in\{1,2,\ldots,4m-1\}\setminus\{2m\}$ then we have $4m-1\nmid 2j-1$, which together with the fact $\gcd(2m-1,4m-1)=1$ implies that  $h(j)$ is not an integer. This proves Claim 3.

Suppose that $n\not\equiv 0\pmod{4}$. Using \eqref{exisimp} together  with Claims 1 and 2, we have
\begin{equation}\label{exin4}
\prod_{j=1}^{n-1}e^\T\xi_j=\begin{cases}
\pm 2^{1-\frac{n}{2}}&\text{if~$n\equiv 2\pmod{4}$},\\
\pm 2^{\frac{1}{2}-\frac{n}{2}}&\text{if~$n\equiv 1,3\pmod{4}$},
\end{cases}
\end{equation}
i.e., $\prod_{j=1}^{n-1}e^\T\xi_j= \pm2^{1-\lceil\frac{n}{2}\rceil}$. This together with Claim 3 completes the proof of this lemma.
\end{proof}
We are in a position to present the proof of Theorem  \ref{mainthm}.
\begin{proof}[Proof of Theorem \ref{mainthm}]
 Let $\lambda_1, \lambda_2,\ldots,\lambda_{n-1}$ and $\xi_1,\xi_2,\ldots,\xi_{n-1}$ be the eigenvalues and associated eigenvectors of $B^\T$ as described in Lemma \ref{eigB}. By Lemma \ref{forfirst} and Lemma \ref{fordet}, we have
	$\prod_{1\le k<j\le n-1} (\lambda_j-\lambda_k)=\pm 2^{\frac{n-2}{2}}(n-1)^{\frac{n-1}{2}}$ and $\det [\xi_1,\xi_2,\ldots,\xi_{n-1}]=\pm2^{-\frac{n-2}{2}}(n-1)^{\frac{n-1}{2}}$. 	If $4\nmid n$ then  by Lemma \ref{etxi}, we have $\prod_{j=1}^{n}e^\T\xi_j=\pm 2^{1-\lceil\frac{n}{2}\rceil}.$ It follows from Lemma  \ref{relWA} that
	$$\det W(B)=\frac{\prod_{1\le k<j\le n-1} (\lambda_j-\lambda_k)\prod_{j=1}^{n-1} e^\T\xi_j}{\det[\xi_1,\xi_2,\ldots,\xi_{n-1}]}=\pm 2^{\lfloor\frac{n}{2}\rfloor-1}.$$ Noting that
	$\hat{W}(D)=W(B)$ by Lemma \ref{cofdet}, Theorem \ref{mainthm}(\rmnum{1}) is proved.  Similarly, if $4\mid n$ then we see that $\rank\, W(B)=|\{j\colon\,1\le j\le n-1 ~and~e^\T \xi_j\neq 0\}|=n-2$, that is, $\rank\, \hat{W}(D)=n-2$. This completes the proof.
	\end{proof}
	
	We proceed to prove Eq. \eqref{rkW} and Theorem \ref{snfd}.
	\begin{lemma}\label{equsnf}
		$W(D_n)$ and $\begin{pmatrix}0&0\\\hat{W}(D_n)&0\end{pmatrix}_{n\times n}$ have the same Smith normal form. In particular, $\rank\, W(D_n)=\rank\, \hat{W}(D_n)$.
	\end{lemma}
\begin{proof}
Noting that the first two rows of $W(D_n)$ are equal, we clearly have $\rank\, W(D_n)\le n-1$. Let $r=\rank\, W(D_n)$ and $A=A(D_n)$. Then $A^re$ can be written as a linear
combination of $e, Ae,\ldots, A^{r-1}e$ with integral coefficients.  As $r\le n-1$, we see that  $A^{n-1}e$ can be  written as a linear
combination of $A^{n-1-r}e, A^{n-r}e,\ldots, A^{n-2}e$ with integral coefficients. Thus, using some evident elementary row and column operations on $W(D_n)$ over the integer ring $\mathbb{Z}$, we can
change $W(D_n)$ to the form $\begin{pmatrix}0&0\\\hat{W}(D_n)&0\end{pmatrix}$. This proves the lemma.
\end{proof}
\begin{proof}[Proof  of Eq. \eqref{rkW}] By  Lemma \ref{equsnf}, we have $\rank\, W(D_n)=\rank\, \hat{W}(D_n)$. Now Eq. \eqref{rkW} clearly follows from Theorem \ref{mainthm}.
	\end{proof}
For an integral matrix $M$ of order $n$, we use  $\rank_2 M$ to denote the rank of $W$ over the binary field $\mathbb{Z}/2\mathbb{Z}$.
\begin{lemma}[\cite{wang2006PHD}]\label{rk2}
	Let $G$ be any graph with $n$ vertices. Then $\rank_2 W(G)\le \lceil\frac{n}{2} \rceil$.
In other words, at most $\lceil\frac{n}{2}\rceil$ invariant factors are congruent to 1 modulo 2.
	\end{lemma}

\begin{proof}[Proof of Theorem \ref{snfd}]
	As $4\nmid n$, Theorem \ref{mainthm} indicates that $\det \hat{W}(D_n)=\pm 2^{\lfloor\frac{n}{2}\rfloor-1}$. Thus, by Lemma \ref{equsnf}, we find that the Smith normal form of $W(D_n)$ has the following pattern:
		\begin{equation} \label{generalform}
	\diag[\underbrace{1,1,\ldots,1}_{r},\underbrace{2^{l_1},2^{l_2},\ldots,2^{l_{n-1-r}}}_{n-1-r},0],
	\end{equation}
where $1\le l_1\le l_2\le\cdots\le l_{n-1-r}$ and $l_1+l_2+\cdots+l_{n-1-r}=\lfloor\frac{n}{2}\rfloor-1$. By Lemma \ref{rk2}, we have $r\le\lceil\frac{n}{2}\rceil$ and hence $n-1-r\ge \lfloor\frac{n}{2}\rfloor-1$. Thus, we have
\begin{equation}\label{addl}
l_1+l_2+\cdots+l_{n-1-r}\ge n-1-r\ge \left\lfloor\frac{n}{2}\right\rfloor-1.
\end{equation}
Noting that both equalities \eqref{addl} must hold simultaneously, we must have $l_1=l_2=\cdots=l_{n-1-r}=1$ and $r=\lceil\frac{n}{2}\rceil$. Thus \eqref{generalform} becomes
	\begin{equation} \label{generalform2}
\diag[\underbrace{1,1,\ldots,1}_{\lceil\frac{n}{2}\rceil},\underbrace{2,2,\ldots,2}_{\lfloor\frac{n}{2}\rfloor-1},0].
\end{equation}
This proves Theorem \ref{snfd}.
\end{proof}

	\section*{Acknowledgments}
This work is supported by the	National Natural Science Foundation of China (Grant Nos. 12001006 and 11971406) and the Scientific Research Foundation of Anhui Polytechnic University (Grant No.\,2019YQQ024).

\end{document}